\documentclass[psamsfonts,english,a4paper,10pt]{amsart}
\usepackage[mathscr]{eucal}
\usepackage{amssymb}
\usepackage{amscd}
\usepackage{babel}
\usepackage{color}

\usepackage{hyperref}
\renewcommand{\phi}{\varphi}

\usepackage{amsthm,amssymb,amsmath,enumerate}

\newcommand{\english}{
\theoremstyle{plain}
\newtheorem{theorem}{Theorem}[section]
\newtheorem*{theorem*}{Theorem}
\newtheorem{lemma}[theorem]{Lemma}
\newtheorem{proposition}[theorem]{Proposition} 
\newtheorem*{proposition*}{Proposition}
\newtheorem{corollary}[theorem]{Corollary} 
\newtheorem*{corollary*}{Corollary}
\newtheorem{conjecture}[theorem]{Conjecture} 
\theoremstyle{definition}
\newtheorem{definition}[theorem]{Definition} 
\newtheorem*{definition*}{Definition}
\newtheorem{example}[theorem]{Example}
\newtheorem{exercise}[theorem]{Exercise}
\newtheorem{assumption}[theorem]{Assumption}
\newtheorem{notation}[theorem]{Notation}
\newtheorem*{problem}{Problem}
\theoremstyle{remark}
\newtheorem{remark}{Remark}
\newtheorem{case}{Case}
\newtheorem*{claim}{Claim}
}

\newcommand{\beq}{\begin{equation}}
\newcommand{\eeq}{\end{equation}}
\newcommand{\beqn}{\begin{equation}}
\newcommand{\eeqn}{\end{equation}}
\newcommand{\beqa}{\begin{eqnarray}}
\newcommand{\eeqa}{\end{eqnarray}}
\newcommand{\beqan}{\begin{eqnarray}}
\newcommand{\eeqan}{\end{eqnarray}}
\newcommand{\ben}{\begin{enumerate}}
\newcommand{\een}{\end{enumerate}}
\newcommand{\bit}{\begin{itemize}}
\newcommand{\eit}{\end{itemize}}
\newcommand{\bce}{\begin{center}}
\newcommand{\ece}{\end{center}}
\newcommand{\bsli}{\outlineonly{\begin{slide}}}
\newcommand{\esli}{\outlineonly{\end{slide}}}
\newcommand{\ignore}[1]{}

 \english
 \newcommand{\D}{\mathbb D}
 
\begin{document}
\numberwithin{equation}{section}

\title[Semi-Bloch functions in several complex variables]
{Semi-Bloch functions in \\ several complex variables}

\author[U. Backlund]{Ulf Backlund}

\address{Ulf Backlund \\ Danderyds Gymnasium \\ 
 S-182 36 Danderyd,  Sweden}
\email{ulf.backlund@dagy.danderyd.se} 
 
 \author[L. Carlsson]{Linus Carlsson}
 
 \address{Linus Carlsson \\ Academy of Culture and Communication \\  M\"{a}lardalen University \\ Box 883, S-721 23 V\"{a}ster\aa{}s, Sweden}
\email{linus.carlsson@mdh.se}

\author[A. F\"allstr\"om]{Anders F\"allstr\"om}

\address
{Anders F\"allstr\"om \\ Department of Mathematics and mathematical Statistics \\ Ume\aa{} University \\
S-901 87 Ume{\aa},  Sweden}
\email{Anders.Fallstrom@math.umu.se}

\author[H. Persson]{H\aa{}kan Persson}
\address
{H\aa{}kan Persson \\ Department of Mathematics \\Uppsala University \\ S-751 06 Uppsala, Sweden}\email{hakan.persson@math.uu.se}

\begin{abstract}
Let $M$ be an $n$-dimensional complex manifold. A holomorphic function 
$f:M\to {\mathbb C}$ is said to be semi-Bloch if for every 
$\lambda\in{\mathbb C}$ the function $g_{\lambda}=\text{exp}(\lambda f(z))$ is 
normal on $M$. We characterise Semi-Bloch functions on infinitesimally Kobayashi non-degenerate $M$ in geometric as well as analytic terms. Moreover, 
we show that on such manifolds, semi-Bloch functions are normal.
\end{abstract}

\maketitle

\section{Introduction}

The purpose of this paper is to define and investigate semi-Bloch 
functions on $n$-dimensional complex manifolds.  The class of semi-Bloch functions 
on the unit disk in the complex plane form a a natural class between the Bloch functions and the normal functions. They were introduced by F. 
Colonna in \cite{Colonna}, when studying the normality of functions omitting two distinct values. The study was continued by Aulaskari and Lappan, see \cite{AL}, who managed to characterise the class of semi-Bloch functions in the one-dimensional unit disk in analytic as well as geometric terms. S. G. Krantz, see in \cite[p. 121]{KrantzBok}, made a remark on the higher dimensional generalisation of this class and later, M. R. Pouryayeval initiated the study of semi-Bloch functions in higher dimensions. He considered the class of semi-Bloch functions on bounded domains in $\mathbb C ^n$, and managed to extend the analytic characterisation of Aulaskari and Lappan such domains. However, the geometrical characterisation was left open, \cite[p. 65]{P}.

In this paper, we continue this line of research and extend the full Aulaskari-Lappan characterisation to a large class of $n$-dimensional complex manifolds which includes all bounded domains in $\mathbb C^n$, but also many unbounded domains and all Kobayashi hyperbolic complex manifolds. On the way, we show that some fundamental properties of Bloch functions and normal  functions are readily extended to above-mentioned class of complex manifolds. 
 

\section{Definitions and elementary properties.}
 
Throughout this paper, $M$ will denote an $n$-dimensional complex manifold and  
$\hat{\mathbb C}$ will denote the Riemann sphere.
For a point 
$z\in M$ we denote by $T_{z}(M)$ the complex tangent space 
of $M$ at $z$.
Notice that for a given point $\zeta\in{\mathbb C}$, the tangent space 
$T_\zeta({\mathbb C})$ can in a natural way be identified with 
$T_\zeta(\hat{\mathbb C})$.
For $z\in M$, $N$ a complex manifold and $f: M \to N$ a holomorphic mapping, we denote by $f_*(z)$ the differential mapping $T_z(M)\to 
T_{f(z)}(N)$.

In this paper the unit disk in the complex plane will be denoted by 
$\D$ and a holomorphic mapping from $\D$ to $M$ will be called an analytic disk in $M$. For such a mapping $\phi$, we denote by $\phi'(\zeta)$ the tangent vector $\phi'(\zeta)=\phi_*(\zeta)(e)$, where $e$ is the canonical basis of $T_0\mathbb D \cong \mathbb C$.

For each $z \in M$ and $\xi \in T_z(M)$, the infinitesimal form of the Kobayashi-Royden pseudometric, $k_{M}(z,\xi)$, is defined by
\[
 k_{M}(z,\xi) = \inf\bigg\{\frac{1}{|a|} \biggm| 
\phi 
 \text{ is 
an analytic disk}, 
 \phi(0)=z,
 \phi'(0)=a\xi \bigg\}.
\]
For more background information on the Kobayashi-Royden pseudometric, see \cite{JP}, \cite{Kob} and \cite{R}. 

\begin{definition}
	A complex manifold $M$ is called \emph{infinitesimally Kobayashi non-degenerate} if 
	\[
		k_M(z,\xi)>0, \quad \forall z \in M, \xi \in T_z(M)\setminus \{0\}.
	\]
\end{definition}
It was proved by Royden, see \cite{R}, that all Kobayashi hyperbolic complex manifolds are infinitesimally Kobayashi non-degenerate; note however that there are infinitesimally Kobayashi non-degenerate domains in $\mathbb C^n$ which are not hyperbolic, see \cite[Remark 3.5.11]{JP}. The most elementary example of a Kobayashi hyperbolic complex manifold is a bounded domain in $\mathbb C^n$, as follows from the Cauchy estimates. Most of the results of this paper, including condition $(ii]$ of the main theorem, are new already in this setting.

The infinitesimal form of the
spherical metric on $\hat{\mathbb C}$ is defined by
$$\chi(z,\xi)=\frac{2|\xi|}{1+|z|^2}$$
for $z\in \hat{\mathbb C}$ and $\xi\in T_z(\hat{\mathbb C})$.
A space $X$ endowed with a pseudometric $\rho$ will be denoted by 
$(X,\rho)$.

The notion of a schlicht disk is central to the
study of Bloch functions. In higher dimension the appropriate 
definition is as follows:

\begin{definition}
Let $f$ be a holomorphic 
function on $M$. A disk $\D(w_0,r)=\{w\in\mathbb C:|w-w_0|<r\}$ 
is said to be a \emph{schlicht disk in the range of $f$}, if there exists a 
holomorphic function $h:\D\to M$, such that $f\circ h:\D\to \D(w_0,r)$ is biholomorphic.
\end{definition}
Note that for $M=\mathbb C$ this definition coincides with the classical 
definition of a schlicht disk.

The least upper bound of the radii of schlicht disks in the range of a 
holomorphic function $f$ will play a significant role. We denote by 
\[
  d_f(z)=\sup\{r: \D(f(z),r) \text{ is a schlicht disk in the range of 
} f\}.
\] 
\section{Bloch and Normal functions}\label{Blochavsnitt}

In this section we recall the definitions and some facts about Bloch 
and normal functions in higher dimension.
\begin{definition}
A holomorphic function $f:M\to{\mathbb C}$ 
is said to be Bloch if the mapping $f_*(z)$ 
is bounded from $(T_z(M),k_M)$ to $(T_{f(z)}({\mathbb 
C}),|\ |)$, uniformly in $z\in M$.
In other words, $f$ is Bloch if there exists a constant $C$ such that
\begin{equation}\label{eq:bloch_defn}
|f_*(z)(\xi)|\le C k_{M}(z,\xi)
\end{equation}
for all $z\in M$ and  
$\xi\in T_{z}(M)$
\end{definition}
The set of all Bloch functions on $M$ is denoted by
${\mathcal B}(M)$. It is obvious from the definition that $\mathcal B(M)$ has the natural structure of a vector space,  and in some cases it can even be given the structure of a Banach space; see for example \cite{ACP} for the classical case or \cite{T1} and \cite{T2} for an extension to $\mathbb C^n$.

Bloch functions can be characterised by the size of schlicht disks in their range. The following theorem was previously proved for bounded homogeneous domains by Timoney, see \cite{T1}, and by Krantz and Ma, see \cite{KraMa}, for strongly pseudoconvex domains.

\begin{theorem}
	Suppose that $f$ is a holomorphic function on an infinitesimally Kobayashi non-degenerate $M$. Then the following are equivalent:
		\begin{enumerate}[(i)]
			\item
				$f$ is Bloch;
			\item
				$\displaystyle\sup_{z \in M}d_f(z) < \infty$.
		\end{enumerate}
\end{theorem}
We divide the proof of the theorem into two lemmas of independent interest.
\begin{lemma}\label{olikhet}
Let $f$ be a 
holomorphic function on $M$. Then 
\[d_f(z)\le \sup\left\{\frac{|f_*(z)(\xi)|}{k_{M}(z,\xi)} : \xi\in T_{z}(M)\right\}.\]
\end{lemma}
\begin{proof}
Let $z$ be an arbitrary point in $M$ and
let $\D(w_{0},r)=\{w\in{\mathbb C}: 
|w-w_{0}|<r\}$ be a schlicht disk in the range of $f$ centred around 
$w_{0}=f(z)$ with radius $r$. 
This means that there is a 
holomorphic function $h$ from the unit disk in ${\mathbb C}$ to 
$M$ 
such that $g(\zeta)=f\circ h(\zeta)=w_{0}+r\zeta$. 

Choose $\xi_{0}=h'(0)\neq 0$. Then it follows from the definition of the Kobayashi-Royden pseudometric that
 $k_{M}(z,\xi_0)\leq1$, and hence it follows that 
\begin{align*}
	\sup\left\{\frac{|f_*(z)(\xi)|}{k_{M}(z,\xi)} :\xi\in T_{z}(M)\right\} &\ge \frac{|f_*(z)(\xi_{0})|}{k_{M}(z,\xi_{0})}
\\
	&\geq {|f_*(z)(\xi_{0})|}=|f_*(z)(h'(0))|=|g'(0)|=r.
\end{align*}
\end{proof}
\begin{lemma}\label{lem:bloch_konst}
	There is a universal constant $C$ such that whenever $f$ is a holomorphic function on an infinitesimally Kobayashi non-degenerate $M$ and $z \in M$ it holds that
	\[
		\sup\left\{\frac{|f_*(z)(\xi)|}{k_{M}(z,\xi)}: \xi \in T_{z}(M)\right\} \leq C \sup\{d_f(w): w \in M\}.
	\]
\end{lemma}
\begin{proof}
	Let $z \in M$ and $\xi \in T_{z}(M)\setminus\{0\}$ be arbitrary. Since by assumption $k_M(z, \xi) > 0$, there exists for every $\epsilon >0$ an analytic disk $\phi$ in $M$ such that $\phi(0)=z$, $\phi'(0)=a_0\xi$ for some $a_0 \in \mathbb R$ satisfying
	\[
		\frac{1}{a_0} \leq (1+\epsilon)k_M(z,\xi).
	\]
	This means that
	\begin{align*}
		\frac{|f_*(z)(\xi)|}{k_{M}(z,\xi)}
		&=\frac{1}{a_0}\frac{\bigl|(f\circ\phi)'(0)\bigr|}{k_{M}(z,\xi)}\\
		&\leq(1+\epsilon)\bigl|\bigl(f\circ\phi\bigr)'(0)\bigr|.
	\end{align*}
	To continue we put $g=f \circ \phi$, and use Bloch's theorem (see for example \cite[p. 112]{Schiff} to deduce that there is a constant $B$ (Bloch's constant) such that 
	\[
		\sup \left\{ d_g(\zeta): \zeta \in \D\right\}\geq|g'(0)|B.
	\]
	Since $\sup\{d_g(\zeta): \zeta \in \D\} \leq \sup\{d_f(w):w \in M\}$, and since $\epsilon$ was chosen arbitrarily we are finished.
\end{proof}
We now define what it means to be normal for a meromorphic function 
on $M$ with values on the Riemann sphere.

\begin{definition}
A meromorphic function 
$f:M \to \hat{\mathbb C}$ is said to be normal if 
the mapping $f_*(z)$ is bounded from 
$(T_z(M),k_M)$ to $(T_{f(z)}(\hat{\mathbb C}),\chi)$, 
uniformly in $z\in M$. In other words, $f$ is normal if there exists a 
constant $C$ such that 
$$\chi\bigl(f(z),f_*(z)(\xi)\bigr)\le C k_{M}(z,\xi)$$
for all $z\in M$ and all $\xi\in\mathbb C^n$
\end{definition}

The set of all normal functions on $M$ is denoted by
${\mathcal N}(M)$. As opposed to the Bloch functions, $\mathcal{N}(M)$ is in 
general not a vector space, see (\cite{Lappan}). Nevertheless, the space $\mathcal N(M)$ is rather big. For example, any  meromorphic function $f:M\to\hat{\mathbb C}$ with the property that 
$\hat{\mathbb C}\setminus f(M)$ contains at least three points is 
normal, as was shown by J. A. Cima and S. G. Krantz, \cite{Cima}. Their proof is based on the following nice characterisation of normal functions in higher dimensions, which relates the notion of a normal function to that of a normal family of meromorphic functions. For our purposes, we need to extend this proposition to Kobayashi non-degenerate complex manifolds. Although this extension is rather straight-forward, we include the proof for the reader's convenience.
\begin{proposition}\label{prop:normal_disks}
	Suppose that $M$ is Kobayashi non-degenerate and that $f: M \to \hat{\mathbb C}$ is a meromorphic function. Then the following statements are equivalent:
	\begin{enumerate}[(i)]
		\item
			$f$ is normal;
		\item
			for every sequence $\{\phi_j\}$ of analytic disks in $M$, it holds that $\{f\circ \phi_j\}$ is a normal family.
	\end{enumerate}
\end{proposition}
\begin{proof}
	We will follow the proof of Cima and Krantz of the corresponding result for a bounded domain in $\mathbb C^n$, see \cite[Proposition 1.4]{Cima}. For the implication $(i) \Rightarrow (ii)$, we note that it follows from the definition of the Kobayashi-Royden pseudometric that for each analytic disk $\phi$ in M,
\[
	k_M\bigl(\phi(\zeta),\phi'(\zeta)\bigr)\leq k_\mathbb{D}(\zeta,e).
\]
Furthermore, elementary calculations (see for example \cite[Example 5]{KrantzEx}) show that $k_\mathbb{D}(\zeta,e)=1/(1-|\zeta|^2)$. Putting these facts together we get that
\begin{equation}\label{eq:kobayashi_poincare}
	k_M\bigl(\phi(\zeta),\phi'(\zeta)\bigr)\leq \frac{1}{1-|\zeta|^2}.
\end{equation}
Using this and the fact that $f$ is normal we see that
	\[
		\frac{|\bigl(f\circ \phi\bigr)'(\zeta)|}{1+\bigl|f\circ\phi(\zeta)\bigr|^2}\leq C k_\mathbb{M}\bigr(f\bigl(\phi(\zeta)\bigr),\phi'(\zeta)\bigl) \leq C\frac{1}{1-|\zeta|^2}.
	\]
	Since the right-hand side of this estimate is bounded on each $G \Subset \mathbb D$, it follows from Marty's theorem, see for example  \cite[p. 75]{Schiff}, that the family $\{f\circ \phi: \mbox{ $\phi$ is an analytic disk in $M$}\}$ is a normal family.
	
	For the reverse implication, note that it follows from Marty's theorem that  there is a constant $C$ such that
	\begin{align}\label{eq:analdisk_norm}
		\frac{|\bigl(f\circ \phi\bigr)'(0)|}{1+\bigl|f\circ\phi(0)\bigr|^2}\leq C,
	\end{align}
	for each analytic disk $\phi$ in $M$. Now let $z \in M$ and $\xi \in T_z(M)\setminus\{0\}$ be arbitrary and let us proceed as in the proof of Lemma \ref{lem:bloch_konst}. Since $k_M(z,\xi) > 0$ by assumption, there exists for every $\epsilon >0$ an analytic disk $\phi$ in $M$ such that $\phi(0)=z$, $\phi'(0)=a_0\xi$ for some $a_0 \in \mathbb R$ satisfying
	\[
		\frac{1}{a_0} \leq (1+\epsilon)k_M(z,\xi).
	\]
	Combining this with \eqref{eq:analdisk_norm}, we get that
	\[
		\frac{|f_*(z)(\xi)|}{k_M(z,\xi)}=\frac{1}{a_0}\frac{\bigl|\bigl(f\circ \phi\bigr)'(0)\bigr|}{k_M(z,\xi)}\leq (1+\epsilon)\bigl|\bigl(f\circ \phi\bigr)'(0)\bigr| \leq (1+\epsilon)C\Bigl(1+\bigl|f(z)\bigr|^2\Bigr),
	\]
	and since $\epsilon >0$ was arbitrary we are finished.
	\end{proof}
	\begin{remark}
		Note that the assumption of non-degeneracy of $k_M$ only was needed to prove the implication $(ii) \Rightarrow (i)$.
	\end{remark}
We close this section with a generalisation of Lappan's Five-Point Theorem. This theorem follows from a more general theorem of Joseph and Kwack, see \cite{JK}, who studied normal mappings in a much more abstract setting, but for the reader's convenience we provide a proof for the theorem in our setting.
\begin{theorem}\label{thm:Lappan}
	Suppose that $M$ is infinitesimally Kobayashi non-degenerate and that $f: M \to \hat{\mathbb C}$ is a meromorphic function and there is a set $E$ with at least five points such that
		\[
			\sup\left\{\frac{\left|f_*(z)(\xi)\right|}{\bigl(1+ |f(z)|^2\bigr)k_M(z,\xi)}: f(z)\in E, \xi \in \mathbb C^n \right\}< \infty.
		\]
	Then $f$ is normal.
\end{theorem}
\begin{proof}
	By Proposition \ref{prop:normal_disks}, it suffices to show that $\{f\circ \phi: \phi \mbox{ is an analytic disk in $M$}\}$ is a normal family. By a result of Lappan, see \cite{L} or \cite[p. 219]{Z}, it is enough to show that for each $G \Subset \mathbb D$, there is a constant $C$ (possibly depending on $G$) such that
	\begin{equation}\label{eq:lappan}
		\sup\left\{\frac{|\bigl(f\circ \phi\bigr)'(\zeta)|}{1+|f\circ \phi(\zeta)|^2}: \zeta \in G, f\circ \phi(\zeta) \in E\right\} < C,
	\end{equation}
	for each analytic disk $\phi$ in $M$. We now proceed as in the proof of Proposition \ref{prop:normal_disks} and suppose that $G$ is compactly contained in $\mathbb D$, $\zeta \in G$ and that $\phi$ is an analytic disk in $M$ such that $\phi(\zeta) \in E$. Then it holds that
	\[
		\frac{|\bigl(f\circ \phi\bigr)'(\zeta)|}{1+|f\circ \phi(\zeta)|^2}=\frac{\left|f_*\left(\phi(\zeta)\right)\left(\phi'(\zeta)\right)\right|}{1+|f\circ \phi(\zeta)|^2}\leq C k_M\bigl(\phi(\zeta),\phi'(\zeta)\bigr),
	\]
	where $C$ is some constant whose existence is guaranteed by the assumptions of the theorem. We can now deduce from \eqref{eq:kobayashi_poincare} that there is a $C$ (depending on $G$) such that \eqref{eq:lappan} hold. Hence we are finished.
\end{proof}
\section{Semi-Bloch functions}\label{sec:semi-bloch}
In this section we investigate semi-Bloch functions on  
complex manifolds. The main result is a characterisation of 
such functions in geometric as well as analytic terms on infinitesimally Kobayashi non-degenerate complex manifolds. 

It is obvious from the definition that any Bloch function is normal. Furthermore 
if $f$ is Bloch on $M$, 
then for every $\lambda\in{\mathbb C}$ the function
$g=\text{exp}(\lambda f)$ is normal since, by the chain rule,
\begin{align*}
	\chi\bigl(g(z),g_*(z)(\xi)\bigr)&=\frac{2\left|g_*(z)(\xi)\right|}{1+|g(z)|^2}=
\frac{2|\lambda|  |g(z)| \left|f_*(z)(\xi)\right|}{1+|g(z)|^2}\\
&\le
|\lambda|  \left|f_*(z)(\xi)\right|\le Ck_M(z,\xi).
\end{align*}
However, there exist non-Bloch functions $f$ that also have the property 
that for every $\lambda\in{\mathbb C}$ $g=\text{exp}(\lambda f)$ is normal.
This was proved by F. Colonna in \cite{Colonna} who coined the term 
semi-Bloch for such functions on the unit disk. 

\begin{definition}
A holomorphic function $f:M\to{\mathbb C}$ 
is said to be semi-Bloch if for every $\lambda\in{\mathbb C}$ 
the function $g_\lambda:M\to{\mathbb C}$
defined by
$g_\lambda(z)=\exp\bigl(\lambda f(z)\bigr)$ is normal. 
\end{definition}
The set of all semi-Bloch functions on $M$ is denoted by
$\mathcal{SB}(M)$. As in the case of normal functions, 
$\mathcal{SB}(M)$ is in general not a vector space. This was 
proved by R. Aulaskari and P. Lappan in \cite{AL}. In addition they 
showed that semi-Bloch functions are not closed under multiplication.

Colonna showed that semi-Bloch functions on  $\D$ are normal. Her proof uses the fact that the image of a holomorphic 
function on $\D$ is simply connected and that any holomorphic logarithm  of a 
nonvanishing holomorphic normal function is normal, see \cite[Theorem 3]{Colonna}. For semi-Bloch functions on arbitrary complex manifolds, this approach does not apply, not even in $\mathbb C$.  However, using the the following characterisation of semi-Bloch functions, we show that semi-Bloch functions on infinitesimally Kobayashi non-degenerate manifolds are normal.

The main result of this paper is the  generalisation of the main theorem of \cite{AL} to infinitesimally Kobayashi non-degenerate complex manifolds.
\begin{theorem}\label{Satsen}
Suppose that $M$ is infinitesimally Kobayashi non-degenerate and that $f$ is holomorphic on $M$. Then the following 
are equivalent:
\begin{enumerate}[(i)]
\item
$f$ is semi-Bloch in $M$;
\item
For each line $L$ in the complex plane, 
\[
{M}_{L}=\sup\left\{\frac{|f_*(z)(\xi)|}{k_{M}(z,\xi)}  : 
z\in M, f(z)\in L,\xi\in T_{z}(M)\right\}<\infty;
\]
\item
For each line $L$ in the complex plane,
\[
	{D}_{L}=\sup\{d_{f}(z): f(z)\in L\} <\infty.
\]
\end{enumerate}
\end{theorem}
\begin{remark}
	In particular, the characterisation in Theorem 4.2 holds for arbitrary bounded domains in $\mathbb{C}^n$. In this context, see also \cite{P}.
\end{remark}
We prove the theorem by splitting up the proof of the different implications into separate lemmas.
\begin{proposition}\label{summa}
Suppose that $f$ is holomorphic in $M$. Then $f$ is 
semi-Bloch if and only if $f+c$ is semi-Bloch for every constant $c$.
\end{proposition}

\begin{proof}
If $g(z)$ is normal, then $h(z)=cg(z)$ is normal for every constant 
$c$.
This follows immediately from the definition:
\[
	\chi\bigl(h(z),h_*(z)(\xi)\bigr)=
\chi\bigl(h(z),(cg)_*(z)(\xi)\bigr)
=2|c|\frac{|g_*(z)(\xi)|}{1+|cg(z)|^2}\le C\cdot 
k_{M}(z,\xi),
\]
since $g$ is normal.
It follows that if $f$ is semi-Bloch, then 
$\text{exp}({\lambda(f(z)+c)})=e^{\lambda c}\text{exp}({\lambda
f(z)})$ is 
normal for every $\lambda\in{\mathbb C}$ and hence $f(z)+c$ is 
semi-Bloch.

The other implication is obvious.
\end{proof}
\begin{lemma}\label{lem:SB_to_ML}
	Suppose that $f$ is semi-Bloch in $M$. Then
	\[
		{M}_{L}=\sup\left\{\frac{|f_*(z)(\xi)|}{k_{M}(z,\xi)}  : 
z\in M, f(z)\in L,\xi\in T_{z}(M)\right\}<\infty,
	\]
	for each line $L$ in the complex plane. 
\end{lemma}
\begin{proof}
If necessary we can choose complex numbers 
$\lambda$ and $\alpha$ such that $\{w\in{\mathbb C} : w=\lambda 
z+\alpha, z\in L\}$ is 
the imaginary axis. Since $f$ is supposed to be semi-Bloch, 
it follows from Proposition \ref{summa}  
that $\lambda f(z)+\alpha$ is semi-Bloch. We may therefore assume 
that $L$ is the imaginary axis.
Define the function $g(z)=\text{exp}({f(z)})$. Then 
$g$ is normal in $M$
since $f$ is semi-Bloch.
This means that
\[
	\frac{\chi\bigl(g(z),g_*(z)(\xi)\bigr)}{k_{M}(z,\xi)}<\infty,
\]
and since $|g(z)|=1$ for all $z\in M$ such that $f(z)\in L$ 
it therefore follows that for such $z$
\begin{align*}
	\frac{\chi\bigl(g(z),g_*(z)(\xi)\bigr)}{k_{M}(z,\xi)}&= \frac{2|g_*(z)(\xi)|}{\left(1+|g(z)|^2\right)
k_{M}(z,\xi)} = \frac{|g_*(z)(\xi)|}{k_{M}(z,\xi)}\\
	&=\left|\text{exp}({f(z)})\right|\frac{|f_*(z)(\xi)|}{k_{M}(z,\xi)}=
\frac{|f_*(z)(\xi)|}{k_{M}(z,\xi)},
\end{align*}
and therefore $M_L < \infty$.
\end{proof}
\begin{lemma}\label{lem:ML_to_SB}
	Suppose that $f$ is a holomorphic function on an infinitesimally Kobayashi non-degenerate M such that for each line $L \subset \mathbb C$,
	\[
		{M}_{L}=\sup\left\{\frac{|f_*(z)(\xi)|}{k_{M}(z,\xi)}  : 
z\in M, f(z)\in L,\xi\in T_{z}(M)\right\}<\infty.
	\]
	Then $f$ is semi-Bloch.
\end{lemma}
\begin{proof}
	We want to show that $g(z)=\exp\bigl(\lambda f(z)\bigr)$ is normal. Since this is trivially true for $\lambda = 0$,
  we can assume that $\lambda \neq 0$.
  Now let $L=\{\zeta \in \mathbb C: \operatorname {Re}(\lambda\zeta) = 0\}$. Since $|g(z)|=1$ when $f(z)\in L$, we then have that for $z \in M$ such that $f(z) \in L$:
  \[
    \frac{|g_*(z)(\xi))|}{k_{M}(z,\xi)\bigl(1+|g(z)|^2\bigr)}=\frac{|\lambda| |f_*(z)(\xi)|}{2k_{M}(z,\xi)}<\frac 1 2| \lambda| M_L < \infty.
  \]
  It now follows from Theorem \ref{thm:Lappan} that $g$ is normal, and a hence that $f$ is semi-Bloch.
\end{proof}
For the next lemma we need the following version of Ahlfors's Five-Island-Theorem. For an accessible proof of this deep theorem, see for example \cite{Bergweiler}.
\begin{theorem}[Ahlfors]\label{thm:Ahlfors2}
	 Suppose that $G$ is a domain in $\mathbb C$ and that $R_1, R_2$ and $R_3$ are mutually disjoint Jordan domains in ${\mathbb C}$. Then the family
	 \begin{multline*}
	 	\mathcal F=\{f: f \mbox{ is holomorphic in $G$, and $f$ maps no subdomain}\\
		\mbox{ of $G$ conformally onto $R_j$ for $j\in \{1,2,3\}$}. \},
	 \end{multline*}
	 is a normal family.
\end{theorem}
\begin{lemma}\label{lem:DL_to_ML}
Suppose that $M$ is infinitesimally Kobayashi non-degenerate and that $f$ is a holomorphic function on $M$. If there is a line $L$ in $\mathbb C$ such that
\[
{M}_{L}=\sup\left\{\frac{|f_*(z)(\xi)|}{k_{M}(z,\xi)}  : 
z\in M, f(z)\in L,\xi\in T_{z}(M)\right\}=\infty,
\]
then
\[
	{D}_{L}=\sup\{d_{f}(z): f(z)\in L\} =\infty.
\]
\end{lemma}
\begin{proof}
	By Lemma \ref{summa}, we may---as in the proof of Lemma \ref{lem:SB_to_ML}---assume that $L=\mathbb R$. Since $M_L=\infty$, there are points $z_j \in M \subset \mathbb C^n$ and $\xi_j \in T_{z_j}(M)$ such that $f(z_j) \in \mathbb R$ and
	\[
		\lim_{j\to \infty} \frac{|f_*(z_j)(\xi_j)|}{k_{M}(z_j,\xi_j)}=\infty.
	\]
	Since $k_M(z_j,\xi_j)>0$ by assumption, we may proceed as in the proof of Lemma \ref{lem:bloch_konst}, to conclude that there are analytic disks $\phi_j$ in $M$ such that $\phi_j(0)=z_j$ and
	\[
		\lim_{j \to \infty} \bigl(f\circ \phi_j\bigr)'(0)=\infty.
	\]
	By setting $g_j(\zeta)=f \circ \phi_j(\zeta)-f\circ \phi_j(0)$ we get that
	\[
		\lim _{j\to \infty} \frac{g_j'(0)}{1+|g_j(0)|^2}=\lim_{j\to \infty} \bigl(f\circ \phi_j\bigr)'(0) = \infty.
	\]
	By Marty's theorem, see \cite[p. 75]{Schiff}, this means that $g_j$ is not a normal family, and by Theorem \ref{thm:Ahlfors2}, this means that for each fixed $k$, there is a $g_{j_k}$ that maps a subdomain of $\mathbb D$ conformally onto one of the domains $\{w \in \mathbb C: |w-(k+\ell)!|<k+\ell\}$ where $\ell = 0,1,2$.  Since the range of $g_{j_k}$ is just the range of $f\circ \phi_{j_k}$ translated by the real number $f\circ \phi_{j_k}(0)$, this means that $d_f(z_k)\geq d_{f\circ \phi_{j_k}}(0)\geq k$. Since $k$ was arbitrary, we are finished.
\end{proof}
We have now implicitly proved the main theorem. To make this clear we state this as a proof.
\begin{proof}[Proof of Theorem \ref{Satsen}]
	The implications $(i) \Rightarrow (ii) \Rightarrow (iii)$ are in turn proved by Lemma \ref{lem:SB_to_ML} and Lemma \ref{olikhet}. The reverse implications are proved by Lemma \ref{lem:DL_to_ML} and Lemma \ref{lem:ML_to_SB}.
\end{proof}
We end this section with the following proposition clarifying the interrelatedness of the function classes treated in this paper.
\begin{proposition}
For an infinitesimally Kobayashi non-degenerate complex manifold $M$, it holds that
\[
	{\mathcal B}(M)\subset{\mathcal S}{\mathcal 
B}(M)\subset{\mathcal N}(M),
\]
and there are examples where these inclusions are proper.
\end{proposition}
\begin{proof}
We begin with showing the inclusions. We already showed in the beginning of Section \ref{sec:semi-bloch} that Bloch functions are semi-Bloch. To see that a function $f\in \mathcal{SB}(M)$ is normal, it is enough to note that by Lemma \ref{lem:SB_to_ML}, the assumptions of Theorem \ref{thm:Lappan} holds with $E$ any line in $\mathbb C$.

F. Colonna gave an example of a function that is semi-Bloch but not Bloch, see \cite{Colonna}. It follows from Theorem \ref{thm:Lappan} that a conformal mapping from $\D$ to $\mathbb C \setminus \{x+0i \in \mathbb C: x >0\}$ is normal, but by $(iii)$ of Theorem \ref{Satsen}, such a function cannot be semi-Bloch. Of course, these examples can trivially be extended to product domains in $\mathbb C^n$.
\end{proof}
\bibliographystyle{amsalpha}

\end{document}